\documentclass[10pt]{amsart}

\usepackage{amssymb,latexsym, mathtools,tikz}
\usepackage{enumerate}
\usepackage{graphicx}
\usepackage{float}
\usepackage{placeins}
\usepackage{mdframed}
\usepackage{amssymb}
\usepackage{esint}
\usepackage{cool}
\usepackage[all,cmtip]{xy}
\usepackage{mathtools}
\usepackage{amstext} 
\usepackage{array}   
\usepackage[shortlabels]{enumitem}
\usepackage{ytableau}

\newcolumntype{L}{>{$}l<{$}} 
\newtheorem{theorem}{Theorem}[section]
\newtheorem{lemma}[theorem]{Lemma}

\newtheorem{prop}[theorem]{Proposition}
\newtheorem{setup}[theorem]{Setup}
\theoremstyle{definition}
\newtheorem{definition}[theorem]{Definition}
\newtheorem{example}[theorem]{Example}

\newtheorem{obs}[theorem]{Observation}
\newtheorem{notation}[theorem]{Notation}

\theoremstyle{remark}
\newtheorem{remark}[theorem]{Remark}

\newtheorem{the context}[theorem]{The Context}

\numberwithin{equation}{theorem}
\numberwithin{equation}{section}

\newcommand{\id}{\operatorname{id}}

\newcommand{\rank}{\operatorname{rank}}

\newcommand{\coker}{\operatorname{Coker}}

\newcommand{\im}{\operatorname{Im}}

\newcommand{\Ker}{\operatorname{Ker}}

\newcommand{\ideal}[1]{\mathfrak{#1}}
\newcommand{\m}{\ideal{m}}

\newcommand{\bbz}{\mathbb{Z}}
\newcommand{\bbn}{\mathbb{N}}

\renewcommand{\geq}{\geqslant}
\renewcommand{\leq}{\leqslant}
\renewcommand{\ker}{\Ker}

\newcommand{\maps}[5]{\xymatrix{#1 \ar[r]^-{#3} & #2 \\
#4 \ar@{|->}[r] & #5 \\}}
\newcommand{\kos}{\textrm{Kos}}

\newcommand{\ek}{\textrm{EK}}
\newcommand{\lcm}{\textrm{lcm}}

\def\im{\operatorname{im}}

\begin{document}
\title{The DG Products of Peeva and Srinivasan Coincide}



\author{Keller VandeBogert}
\address{University of South Carolina}
\email{kellerlv@math.sc.edu}
\date{\today}

\maketitle

\begin{abstract}
    Consider the ideal $(x_1 , \dotsc , x_n)^d \subseteq k[x_1 , \dotsc , x_n]$, where $k$ is any field. This ideal can be resolved by both the $L$-complexes of Buchsbaum and Eisenbud, and the Eliahou-Kervaire resolution. Both of these complexes admit the structure of an associative DG algebra, and it is a question of Peeva as to whether these DG structures coincide in general. In this paper, we construct an isomorphism of complexes between the aforementioned complexes that is also an isomorphism of algebras with their respective products, thus giving an affirmative answer to Peeva's question.
\end{abstract}

\section{Introduction}

Let $R = k[x_1 , \dotsc , x_n]$ denote a standard graded polynomial ring over a field $k$. Given a homogeneous ideal $I \subseteq R$, let $(F_\bullet,d_\bullet)$ denote a homogeneous minimal free resolution of $R/I$. It is always possible to construct a morphism of complexes $(F \otimes_R F)_\bullet \to F_\bullet$ extending the identity in homological degree $0$; this induces a product $\cdot: F_i \otimes F_j \to F_{i+j}$. Tracing through the definition of the tensor product complex, one finds that this product satisfies the following identity:
$$d_{i+j} (f_i \cdot f_j) = d_i (f_i) \cdot f_j + (-1)^i f_i \cdot d_j (f_j).$$
In general, this product need not be associative (though it is always associative up to homotopy). When this product is associative, we say that $F_\bullet$ admits the structure of an associative DG algebra.

The existence of associative DG algebra structures on minimal free resolutions of cyclic modules is an interesting and often desirable property (see, for instance, \cite{buchsbaum1977algebra} and \cite{avramov1998infinite} for applications of DG techniques). However, one does not have to go far to find ideals for which no associative DG algebra structure exists; the ideal $(x_1^2 , x_1 x_2 , x_2 x_3 , x_3 x_4 , x_4^2) \subseteq k[x_1 , x_2 , x_3 , x_4]$ is a standard counterexample (see \cite{avramov1998infinite}, Theorem $2.3.1$). If $R/I$ has projective dimension at most $3$, then it is known that the minimal free resolution admits the structure of an associative DG algebra (see, for instance, \cite{buchsbaum1977algebra}). For quotients with projective dimension at least $4$, further restrictions must be applied to $I$ in order to ensure such a DG structure. For complete intersections, the Koszul complex is a canonical example of an associative DG algebra, where the product is induced by exterior multiplication. Other classes of ideals for which the minimal free resolution of $R/I$ always admits the structure of an associative DG algebra include: grade $4$ Gorenstein ideals (see \cite{kustin1980algebra}, \cite{kustin1987gorenstein}, and \cite{kustin2019resolutions}), grade $4$ almost complete intersection ideals (see \cite{kustin1994complete}), Borel ideals (see \cite{peeva19960}), matroidal ideals (see \cite{skoldberg2011resolutions}), edge ideals of cointerval graphs (see \cite{skoldberg2016minimal}), ideals of maximal minors (when $k$ has characteristic $0$), and powers of complete intersection ideals (see \cite{srinivasan1989algebra} for both of the previous cases).

Notice that if $R = k[x_1 , \dots , x_n]$ and $\m = (x_1 , \dots, x_n)$, then the previous paragraph says that the minimal free resolution of $R/\m^d$ for any $d$ can be given the structure of a DG-algebra in two ways: first, $\m^d$ is a stable ideal, so one can use the algebra structure provided by Peeva \cite{peeva19960}. However, $\m^d$ is also a power of a complete intersection, whence one can use the algebra structure provided by Srinivasan in \cite{srinivasan1989algebra}. In \cite{peeva2010graded} (right under Open Problem $31.4$), Peeva asks whether or not these two algebras coincide; that is, does there exist an isomorphism of complexes that is also an isomorphism of algebras with respect to both of these products? 

In this paper, we answer this question in the affirmative. To prove this, we first introduce a reformulation of the Eliahou-Kervaire resolution in terms of Young tableaux. Using this reformulation, we are able to construct an explicit isomorphism of complexes between the Eliahou-Kervaire resolution and the $L$-complex of Buchsbaum and Eisenbud. As it turns out, we will prove that this isomorphism of complexes is also an isomorphism of algebras with respect to the algebra structures constructed by Peeva and Srinivasan.

The paper is organized as follows. In Section \ref{sec:LcompandDG}, we recall the $L$-complexes of Buchsbaum and Eisenbud and the associative DG algebra structure constructed by Srinivasan in \cite{srinivasan1989algebra}. We also introduce some notation for Young tableaux that will be used throughout the rest of the paper. In Section \ref{sec:EKwithTableaux}, we first reformulate the Eliahou-Kervaire resolution of $R/(x_1 , \dotsc , x_n)^d$ with Young tableaux (see \ref{def:EKfortableaux}). We then recall the DG structure constructed by Peeva on the Eliahou-Kervaire resolution.

In Section \ref{sec:isoEKandL}, we construct an explicit isomorphism between the Eliahou-Kervaire resolution and the $L$-complex (see Proposition \ref{prop:isoofEKandL}). Finally, in Section \ref{sec:theSame}, we show that this isomorphism of complexes is also an isomorphism of algebras with respect to the algebra structures introduced in Section \ref{sec:LcompandDG} and \ref{sec:EKwithTableaux}.

\section{L-complexes and the Associated DG Structure}\label{sec:LcompandDG}

In this section we introduce the first complexes of interest, namely, the $L$-complexes originally introduced by Buchsbaum and Eisenbud in \cite{buchsbaum1975generic}. The material up until Proposition \ref{prop:resnofpower}, along with proofs, can be found in \cite{buchsbaum1975generic} or Section $2$ of \cite{el2014artinian}. After this, we introduce some notation for Young tableaux that will turn out to be convenient for later sections. We then define the DG structure on these complexes constructed by Srinivasan in \cite{srinivasan1989algebra}.

\begin{setup}\label{set:Lcomplexsetup}
Let $F$ denote a free $R$-module of rank $n$, and $S = S(F)$ the symmetric algebra on $F$ with the standard grading. Define a complex
$$\xymatrix{\cdots \ar[r] & \bigwedge^{a+1} F \otimes_R S_{b-1} \ar[r]^-{\kappa_{a+1,b-1}} & \bigwedge^{a} F \otimes_R S_{b} \ar[r]^-{\kappa_{a,b}} & \cdots}$$
where the maps $\kappa_{a,b}$ are defined as the composition
\begin{equation*}
    \begin{split}
         \bigwedge^{a} F \otimes_R S_{b} &\to \bigwedge^{a-1} F \otimes_R F \otimes_R S_{b} \\
         & \to \bigwedge^{a-1} F \otimes_R S_{b+1}
    \end{split}
\end{equation*}
where the first map is comultiplication in the exterior algebra and the second map is the standard module action (where we identify $F = S_1 (F)$). Define
$$L_b^a (F) := \ker \kappa_{a,b}.$$
Let $\psi: F \to R$ be a morphism of $R$-modules with $\im (\psi)$ an ideal of grade $n$. Let $\kos^\psi : \bigwedge^i F \to \bigwedge^{i-1} F$ denote the standard Koszul differential; that is, the composition
\begin{equation*}
    \begin{split}
        \bigwedge^i F &\to F \otimes_R \bigwedge^{i-1} F  \quad \textrm{(comultiplication)} \\
        &\to \bigwedge^{i-1} F \quad \textrm{(module action)} \\
    \end{split}
\end{equation*}
\end{setup}

\begin{definition}\label{def:Lcomplexes}
Adopt notation and hypotheses of Setup \ref{set:Lcomplexsetup}. Define the complex
\begin{equation*}
    \begin{split}
        &L(\psi , b) : \xymatrix{0 \ar[r] & L_b^{n-1} \ar[rr]^-{\kos^\psi \otimes 1} & & \cdots \ar[rr]^{\kos^\psi \otimes 1} & & L_b^0 \ar[r]^-{S_b (\psi)} & R \ar[r] & 0 } \\
    \end{split}
\end{equation*}
where $\kos^\psi \otimes 1  : L_b^a (F) \to L_b^{a-1}$ is induced by making the following diagram commute:
$$\xymatrix{\bigwedge^a F\otimes S_b (F) \ar[rr]^-{\kos^\psi \otimes 1}  & & \bigwedge^{a-1} F \otimes S_b(F)  \\
L_b^a (F) \ar[rr]^-{\kos^\psi \otimes 1} \ar[u] & & L_b^{a-1} (F) \ar[u] \\}$$
\end{definition}

The following Proposition illustrates the importance of the complexes of Definition \ref{def:Lcomplexes}; namely, these complexes minimally resolve the quotient rings defined by powers of complete intersection ideals.

\begin{prop}\label{prop:resnofpower}
Let $\psi: F \to R$ be a map from a free module $F$ of rank $n$ such that the image  $\im (\psi)$ is a grade $n$ ideal. Then the complex $L(\psi ,b)$ of Definition \ref{def:Lcomplexes} is a minimal free resolution of $R/\im (\psi)^b$
\end{prop}
We also have (see Proposition $2.5 (c)$ of \cite{buchsbaum1975generic})
\begin{equation*}
    \begin{split}
        &\rank_R L_b^a (F) = \binom{n+b-1}{a+b} \binom{a+b-1}{a}. \\
    \end{split}
\end{equation*}
Moreover, using the notation and language of Chapter $2$ of \cite{weyman2003}, $L_b^a (F)$ is the Schur module $L_{(a+1,1^{b-1})} (F)$. This allows us to identify a standard basis for such modules.

\begin{notation}
We use the English convention for partition diagrams. That is, the partition $(3,2,2)$ corresponds to the diagram
$$ 
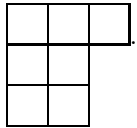
\begin{ytableau}
 \ & \ & \ \\
\ & \ \\
\ & \ \\
\end{ytableau}.$$
A Young tableau is standard if it is strictly increasing in both the columns and rows. It is semistandard if it is strictly increasing in the columns and nondecreasing in the rows.
\end{notation}

\begin{prop}\label{prop:standardbasis}
Adopt notation and hypotheses as in Setup \ref{set:Lcomplexsetup}. Then a basis for $L_b^a (F)$ is represented by all Young tableaux of the form
$$\ytableausetup
{boxsize=2em}
\begin{ytableau}
i_0 & j_1 &\cdots & j_{b-1} \\
i_1 \\
\vdots \\
i_a \\
\end{ytableau}$$
with $i_0 < \cdots < i_{a}$ and $i_0 \leq j_1 \leq \cdots \leq j_{b-1}$. 
\end{prop}

\begin{proof}
See Proposition $2.1.4$ of \cite{weyman2003} for a more general statement.
\end{proof}

\begin{remark}
When viewing the semistandard tableaux of Proposition \ref{prop:standardbasis} as the basis for $L^a_b (F)$, we are tacitly using the fact that $L_b^a (F) = \coker (\kappa_{a+1,b-1} )$.
\end{remark}

The following Observation is sometimes referred to as the \emph{shuffling} or \emph{straightening} relations satisfied by tableaux in the Schur module $L_b^a (F)$.

\begin{obs}\label{obs:straighteningrelations}
Any tableau
$$T= \ytableausetup
{boxsize=2em}
\begin{ytableau}
i_0 & j_1 &\cdots & j_{b-1} \\
i_1 \\
\vdots \\
i_a \\
\end{ytableau}$$
viewed as an element in $L_b^a (F)$ may be rewritten as a linear combination of other tableaux in the following way:
$$T = \sum_{k=0}^a (-1)^k \ytableausetup
{boxsize=1.7em}
\begin{ytableau}
j_1 & i_k & j_2 &\cdots & j_{b-1} \\
i_0 \\
\vdots \\
\widehat{i_k} \\
\vdots \\
i_a \\
\end{ytableau}.$$
Notice that if $i_0 > j_1$ and $i_0 < \cdots < i_a$, then this rewrites $T$ as a linear combination of semistandard tableaux. 
\end{obs}

The following notation will be employed tacitly for the remainder of the paper. It will be a convenient shorthand allowing us to perform computations with tableaux without having to draw the tableaux explicitly.

\begin{notation}\label{not:tableauxnotation}
Let
$$T=\ytableausetup
{boxsize=1.5em}
\begin{ytableau}
i_1 & j_1 &\cdots & j_{b} \\
i_2 \\
\vdots \\
i_a \\
\end{ytableau}$$
denote an arbitrary hook tableau. The notation $T \backslash i_k$ and $T \backslash j_k$ will denote the tableaux
$$T \backslash i_k :=\ytableausetup
{boxsize=1.5em}
\begin{ytableau}
i_1 & j_1 &\cdots & j_{b} \\
\vdots \\
\widehat{i_k} \\
\vdots \\
i_a \\
\end{ytableau}$$
$$T \backslash j_k :=\ytableausetup
{boxsize=1.5em}
\begin{ytableau}
i_1 & j_1 &\cdots & \widehat{j_k} & \cdots & j_{b} \\
\vdots \\
i_a \\
\end{ytableau}.$$
Likewise, given any integer $s \in \bbn$, the notation $T^s$ and $T_s$ will denote the tableaux
$$T^s = \ytableausetup
{boxsize=1.5em}
\begin{ytableau}
i_1 & s & j_1 &\cdots & j_{b} \\
i_2 \\
\vdots \\
i_a \\
\end{ytableau}$$
$$T_s := \ytableausetup
{boxsize=1.5em}
\begin{ytableau}
s & j_1 &\cdots & j_{b} \\
i_1 \\
i_2 \\
\vdots \\
i_a \\
\end{ytableau}.$$
Observe that the above operations descend to well defined homomorphisms on the generators of $L_b^a (F)$. Notice that the following equalities hold in $L_b^a (F)$:
$$(T \backslash i_k )_{i_k} = (-1)^{k-1} T, \quad (T \backslash j_k)^{j_k} = T$$
\end{notation}

\begin{example}
Using the notation of Notation \ref{not:tableauxnotation}, the relation of Observation \ref{obs:straighteningrelations} on the tableau $T$ may be rewritten
$$T = \sum_{k=0}^{a} (-1)^k (T \backslash \{ i_k, j_1 \} )^{i_k}_{j_1}$$
\end{example}

\begin{obs}
Adopt notation and hypotheses of Setup \ref{set:Lcomplexsetup}. With respect to the standard basis elements of $L^a_b (F)$ identified in Proposition \ref{prop:standardbasis}, the differential $\kos^\psi \otimes 1 : L^a_b (F) \to L^{a-1}_b (F)$ takes the form
$$(\kos^\psi \otimes 1) (T) = \sum_{k=0}^a (-1)^k x_{i_k} (T \backslash i_k),$$
where
$$T = \ytableausetup
{boxsize=1.75em}
\begin{ytableau}
i_0 & j_1 &\cdots & j_{b-1} \\
i_1 \\
\vdots \\
i_a \\
\end{ytableau}.$$
\end{obs}

The following observation will turn out to be very helpful in the proof of Proposition \ref{prop:isoofEKandL}. The proof is an immediate consequence of Observation \ref{obs:straighteningrelations}.

\begin{obs}\label{obs:strtrewrite}
Let $T$ be as in Notation \ref{not:tableauxnotation}. Viewing $T$ as an element of $L_b^a (F)$, the following relation holds:
$$(T \backslash i_a)^{i_a} = \sum_{k=1}^{a-1} (-1)^{a-k} (T \backslash i_k)^{i_k}.$$
\end{obs}

Next, we turn to defining the algebra structure on the $L$-complexes of Definition \ref{def:Lcomplexes}. For convenience, we first recall the definition of a DG algebra.

\begin{definition}\label{dg}
A \emph{differential graded algebra} $(F,d)$ (DG-algebra) over a commutative Noetherian ring $R$ is a complex of finitely generated free $R$-modules with differential $d$ and with a unitary, associative multiplication $F \otimes_R F \to F$ satisfying
\begin{enumerate}[(a)]
    \item $F_i F_j \subseteq F_{i+j}$,
    \item $d_{i+j} (x_i x_j) = d_i (x_i) x_j + (-1)^i x_i d_j (x_j)$,
    \item $x_i x_j = (-1)^{ij} x_j x_i$, and
    \item $x_i^2 = 0$ if $i$ is odd,
\end{enumerate}
where $x_k \in F_k$.
\end{definition}

The following definition is the definition of the product on the $L$-complexes originally introduced by Srinivasan. The inductive nature of these products will end up being rather convenient for the proofs in Section \ref{sec:theSame}.

 \begin{definition}\label{def:LDGprod}
Let $1 \leq q_1 \leq \cdots \leq q_d \leq n$, $1 \leq p_1 \leq \cdots \leq p_d \leq n$. Let
$$T = \ytableausetup
{boxsize=1.5em}
\begin{ytableau}
p_1 & p_2 &\cdots & p_d \\
\end{ytableau},$$
$$T' = \ytableausetup
{boxsize=1.5em}
\begin{ytableau}
i_1 & q_2 & q_3 &\cdots & q_d \\
\vdots \\
i_k
\end{ytableau},$$
 be elements of $L(\psi , d)_1$. The DG product of Theorem \ref{thm:LcompisDG} is defined in the following way:
\begingroup\allowdisplaybreaks
\begin{align*}
    L(\psi , d)_1 \otimes L(\psi , d)_1  &\to L(\psi , d)_2 \\
    T \otimes T' &\mapsto x_{p_2} \cdots x_{p_d} (T')_{p_1} + x_{i_1} \big( (T \backslash p_1) \cdot (T \backslash q_1) \big)^{p_1}, \\
    L(\psi , d)_1 \otimes L(\psi , d)_k  &\to L(\psi , d)_{k+1} \\
    T \cdot T' &\mapsto  \begin{cases} x_{p_2} \cdots x_{p_d} (T')_{p_1}  & \textrm{if} \ p_1 \leq i_1, \ i_1 \leq q_2 \\
    x_{p_2} \cdots x_{p_d} (T')_{p_1} +\frac{x_{q_2}}{x_{p_1}} T \cdot (T' \backslash q_2)^{p_1} & \textrm{if} \ p_1 \leq i_1, \ i_1 > q_2, \\
    \frac{-1}{x_{q_2} \cdots x_{q_d}} T'' \cdot (T \cdot (T' \backslash i_1)) & \textrm{otherwise},
    \end{cases} \\
    &\textrm{where} \ T'' = \ytableausetup
{boxsize=1.5em}
\begin{ytableau}
i_1 & q_2 & q_3 &\cdots & q_d \\
\end{ytableau}. \\
\end{align*}
\endgroup
In the above, notice that the products are inductively defined. Finally, if
$$T = \ytableausetup
{boxsize=1.5em}
\begin{ytableau}
j_1 & p_2 &\cdots & p_d \\
\vdots \\
j_\ell
\end{ytableau},$$
one inductively defines:
$$T \cdot T' := \frac{1}{x_{p_2} \cdots x_{p_d}} T'' \cdot \big( (T \backslash j_1) \cdot T' \big),$$
where $T'' = \ytableausetup
{boxsize=1.5em}
\begin{ytableau}
j_1 & p_2 &\cdots & p_d \\
\end{ytableau}$. 
\end{definition}

\begin{theorem}[{\cite[Theorem 3.5]{srinivasan1989algebra}}]\label{thm:LcompisDG}
Adopt notation and hypotheses as in Setup \ref{set:Lcomplexsetup}. Then the complexes $L (\psi , b)$ admit the structure of a commutative, associative DG algebra for all $b \geq 1$. 
\end{theorem}

\section{A Reformulation of Eliahou-Kervaire with Young Tableaux and the Induced DG Structure}\label{sec:EKwithTableaux}

In this section, we reformulate the Eliahou-Kervaire resolution for the ideal $(x_1 , \dotsc , x_n)^d \subseteq k[x_1 , \dotsc , x_n]$ in the language of Young tableaux. This reformulation is for the convenience of defining the isomorphism of complexes of Proposition \ref{prop:isoofEKandL}. We then recall the algebra structure on the Eliahou-Kervaire resolution constructed by Peeva in \cite{peeva19960}. 

\begin{definition}\label{def:EKfortableaux}
Let $R = k[x_1 , \dotsc , x_n]$ be a standard graded polynomial ring over a field $k$. Let 
$$\ek_\bullet : \xymatrix{0 \ar[r] & \ek_n \ar[r]^-{d_\ek} & \cdots \ar[r]^{d_\ek} & \ek_1 \ar[r]^{d_\ek} & R}$$ denote the sequence of $R$-modules and $R$-module homomorphisms with $EK_{k+1}$ the free $R$-module on the basis
$$\ytableausetup
{boxsize=1.5em}
\begin{ytableau}
i_1 & j_1 &\cdots & j_{d} \\
i_2 \\
\vdots \\
i_k \\
\end{ytableau}$$
with $1 \leq i_1 < \cdots < i_k < j_d \leq n$, and $1 \leq j_1 \leq j_2 \leq \cdots \leq j_d$. For $k > 0$, define
$$d_\ek ( T ) = \sum_{\ell=1}^k (-1)^\ell x_{i_\ell} (T \backslash i_\ell) - \sum_{\ell=1}^k (-1)^\ell x_{j_d} (T \backslash \{ i_\ell , j_d \} )^{i_\ell},$$
and for $k=0$,
$$d_\ek  \Bigg(\ytableausetup
{boxsize=1.5em}
\begin{ytableau}
 j_1 &\cdots & j_{d} \\
\end{ytableau} \Bigg) := x_{j_1} \cdots x_{j_d}.$$
\end{definition}

\begin{remark}
We will consider the tableaux comprising the basis elements for each $\ek_k$ to be alternating in the first column and symmetric in the blocks $\ytableausetup
{boxsize=1.5em}
\begin{ytableau}
 j_1 &\cdots & j_{d} \\
\end{ytableau}$.
\end{remark}

\begin{prop}\label{prop:resofmd}
The sequence of $R$-modules and $R$-module homomorphisms of Definition \ref{def:EKfortableaux} forms a complex and is a homogeneous minimal free resolution of the quotient defined by $(x_1 , \dotsc , x_n)^d$. 
\end{prop}

The proof Proposition \ref{prop:resofmd} will follow after identifying an explicit isomorphism with the Eliahou-Kervaire resolution of Borel-fixed ideals. The Eliahou-Kervaire resolution was originally introduced by Eliahou and Kervaire in \cite{eliahou1990minimal}. The definition used here is taken from \cite{peeva2008minimal}.

\begin{definition}\label{def:borelfixed}
Let $R = k[x_1 , \dotsc , x_n]$ be a standard graded polynomial ring over a field $k$. A monomial ideal $I$ is Borel-fixed if for all $g \in I$,
$$g x_j \in I \implies g x_i \in I \ \textrm{whenever} \ i<j.$$
Given an arbitrary monomial $g \in R$, define
$$\max (g) := \max \{ i \mid x_i \ \textrm{divides} \ g \},$$
$$\min (g) := \min \{ i \mid x_i \ \textrm{divides} \ g \}.$$
Let $g \in I$ be any monomial; there is a unique decomposition $g = uv$ where $u \in I$ is a minimal generator and $\max (u) \leq \min (v)$. Given this decomposition, define $b(g) := u$, where $u$ is called the beginning of $g$.
\end{definition}

\begin{definition}[Eliahou-Kervaire Resolution]\label{def:EKingeneral}
Let $I$ denote a Borel ideal in $R=k[x_1 , \dotsc, x_n]$ minimally generated by monomials $m_1 , \dotsc , m_r$. The Eliahou-Kervaire resolution is the sequence of $R$-modules and $R$-module homomorphisms
$$F_I : \xymatrix{0 \ar[r] & (F_I)_n \ar[r]^-{d} & \cdots \ar[r]^{d} & (F_I)_1 \ar[r]^{d} & R}$$
with $(F_I)_{i+1}$ the free $R$-module on basis denoted
$$(m_p ; j_1 , \dotsc, j_i),$$
where $j_1 < \cdots < j_i < \max (m_p)$. Define $R$-module homomorphisms
$$\partial (m_p ; j_1 , \dotsc, j_i) := \sum_{q=1}^i (-1)^q x_{j_q} (m_q ; j_1 , \dotsc , \widehat{j_q} , \dotsc j_i),$$
$$\mu (m_p ; j_1 , \dotsc, j_i)  := \sum_{q=1}^i \frac{m_p x_{j_q}}{b(m_p x_{j_q})} (b(m_q x_{j_q}) ; j_1 , \dotsc , \widehat{j_q} , \dotsc j_i).$$
Define the differential $d : (F_I)_i \to (F_I)_{i-1}$ via
$$d := \partial - \mu.$$
\end{definition}

The following Theorem demonstrates the significance of the complex of Definition \ref{def:EKingeneral}; namely, the Eliahou-Kervaire resolution minimally resolves all Borel ideals. 

\begin{theorem}\label{thm:EKresBorelfixed}
Let $R = k[x_1 , \dotsc , x_n]$ be a standard graded polynomial ring over a field $k$ and let $I \subseteq R$ be a Borel ideal. Then the complex $F_I$ as in Definition \ref{def:EKingeneral} is a minimal free resolution of $R/I$.
\end{theorem}

Observe that the following Proposition immediately implies Proposition \ref{prop:resofmd}.

\begin{prop}\label{prop:EKandTabiso}
Let $R= k[x_1 , \dotsc , x_n]$ and $\m = (x_1 , \dotsc , x_n)$. For each $k$, define $R$-module homomorphisms
\begingroup\allowdisplaybreaks
\begin{align*}
    \eta_{k+1} : \ek_{k+1} &\to (F_{\m^d})_{k+1} \\
    \ytableausetup
{boxsize=1.5em}
\begin{ytableau}
i_1 & j_1 &\cdots & j_{d} \\
i_2 \\
\vdots \\
i_k \\
\end{ytableau} &\mapsto (x_{j_1} \cdots x_{j_d} ; i_1 , \dotsc , i_k ) . \\
\end{align*}
\endgroup
Then $\eta_\bullet$ is an isomorphism of complexes.
\end{prop}

\begin{proof}
The map $\eta_k$ is clearly an isomorphism for each $k$. To see that this is an isomorphism of complexes, simply observe that any minimal generator of $\m^d$ is of the form $x_{q_1} \cdots x_{q_d}$ for $q_1 \leq \cdots \leq q_d$, so that for any $j < q_d$,
$$b ( x_{q_1} \cdots x_{q_d} \cdot x_j) = x_j x_{q_1} \cdots x_{q_{d-1}}.$$
\end{proof}

Next, we need to introduce some necessary definitions in order to define the DG structure on the Eliahou-Kervaire resolution.

\begin{notation}
Let $R= k[x_1 , \dotsc , x_n]$, where $k$ is a field. Given a vector $\alpha = (\alpha_1 , \dotsc , \alpha_n) \in \bbz_{\geq 0}^n$, use the notation
$$x^\alpha := x_1^{\alpha_1} \cdots x_n^{\alpha_n}.$$
The notation $\epsilon_i$ will denote the appropriately sized vector will a $1$ in the $i$th entry and $0$'s elsewhere.
\end{notation}

\begin{definition}\label{def:meetandchains}
Given two minimal generators $x^\alpha$, $x^\beta$ of $(x_1 , \dotsc , x_n)^d$, write
$$x_{s_1} \cdots x_{s_a} = \frac{b(\lcm (x^\alpha,x^\beta))}{\gcd (b(\lcm (x^\alpha,x^\beta),x^\alpha)},$$
with $s_1 \leq \cdots \leq s_a$ and define the sequence of monomials $f_0 , f_1 , \dotsc , f_a$ via $f_0 = x^\alpha$, $f_{i+1} = b(x_{s_{i+1}} f_i)$. 

Likewise, write 
$$x_{t_1} \cdots x_{t_b} = \frac{b(\lcm (x^\alpha,x^\beta))}{\gcd (b(\lcm (x^\alpha,x^\beta),x^\beta)},$$
with $t_1 \leq \cdots \leq t_b$ and define the sequence of monomials $g_0 , g_1 , \dotsc , g_b$ via $g_0 = x^\beta$, $g_{i+1} = b(x_{s_{i+1}} g_i)$. 
\end{definition}

The following Theorem is due to Peeva \cite{peeva19960}. Combining Theorem \ref{thm:EKisDGPeeva} with the isomorphism of Proposition \ref{prop:EKandTabiso} induces a DG algebra structure on the complexes $\ek_\bullet$ of Definition \ref{def:EKfortableaux}.

\begin{definition}\label{def:DGstructureonEK}
Let $I := (x_1 , \dotsc , x_n)^d \subseteq R := k[x_1 , \dotsc , x_n]$, and employ the notation of Definition \ref{def:meetandchains} for the monomials $f$ and $g$. Define a product $(F_I)_i \otimes (F_I)_j \to (F_I)_{i+j}$ via the following formulas:
\begingroup\allowdisplaybreaks
\begin{align*}
    &1.  \qquad (f ; J) \cdot (f ; K) = 0, \\
    &2.  \qquad (f ; J) \cdot (g ; K ) = 0 \ \textrm{if there exist} \ s_p \in J \ \textrm{and} \ t_q \in K, \\
    &3. \qquad (f; J) \cdot (g ; K) = \sum_{i<p} \frac{fg}{f_i f_{i+1}} (f_i ; J) \cdot (f_{i+1} ; K) \\
    &\qquad \qquad \qquad \textrm{if} \ s_{p+1} \in J, \ s_\ell \notin J \ \textrm{for} \ \ell \leq p, \ \textrm{and} \ t_i \notin K \ \textrm{for all} \ i, \\
    &4. \qquad (f; J) \cdot (g ; K) = \sum_{i<q} \frac{fg}{g_i g_{i+1}} (g_i ; J) \cdot (g_{i+1} ; K) \\
    &\qquad \qquad \qquad \textrm{if} \ t_{q+1} \in K, \ t_\ell \notin K \ \textrm{for} \ \ell \leq q, \ \textrm{and} \ s_i \notin J \ \textrm{for all} \ i, \\
    &5. \qquad (f ; J) \cdot (g ; K) = \sum_{i} \frac{fg}{f_i f_{i+1}} (f_i ; J) \cdot (f_{i+1} ; K) \\
    &\qquad\qquad\qquad\qquad + \sum_{i} \frac{fg}{g_i g_{i+1}} (g_i ; J) \cdot (g_{i+1} ; K)  \ \textrm{otherwise,} \\
    &\textrm{where} \ (f_i ; J) \cdot (f_{i+1} ; K) = \begin{cases} 
    0 \qquad \textrm{if} \ s_{i+1} \leq \max (K) \\
    \frac{f_{i+1}}{s_{i+1}} (f_i ; J,s_{i+1} , K) \qquad \textrm{if} \ s > \max (K). \\
    \end{cases} \\
\end{align*}
\endgroup
\end{definition}

\begin{theorem}[{\cite[Theorem 1.1]{peeva19960}}]\label{thm:EKisDGPeeva}
Let $I$ be a Borel-fixed ideal. Then the minimal free resolution $F_I$ of $R/I$ admits the structure of a commutative associative DG algebra, where the product is given by Definition \ref{def:DGstructureonEK}.
\end{theorem}

The following Proposition is an immediate consequence of \cite[Proposition 2.8]{peeva19960} and will be very useful for the proof of Lemma \ref{lem:dgFor1n}.

\begin{prop}\label{prop:EKDGprod}
Adopt notation and hypotheses as in Definition \ref{def:EKfortableaux}.
Let $1 \leq p_1 \leq \cdots \leq p_d \leq n$ and $1 \leq q_1 \leq \cdots \leq q_d \leq n$, with $i_1 < \cdots i_k < q_d$. Define
$$T = \ytableausetup
{boxsize=1.5em}
\begin{ytableau}
p_1 & p_2 &\cdots & p_d \\
\end{ytableau},$$
$$T' = \ytableausetup
{boxsize=1.5em}
\begin{ytableau}
i_1 & q_1 & q_2 &\cdots & q_d \\
\vdots \\
i_k
\end{ytableau},$$
and assume $p_1 \leq q_1$. Then,
$$T \cdot T' = - x_{p_2} \cdots x_{p_d} (T')_{p_1} + x_{q_1} T \cdot (T' \backslash q_1)^{p_1}.$$
\end{prop}

\section{An Isomorphism of Complexes Between $\ek_\bullet$ and $L ( \psi , b)$}\label{sec:isoEKandL}

In this section, we construct an explicit isomorphism of complexes $\ek_\bullet \to L(\psi ,  b)$. This isomorphism will end up being the explicit isomorphism of algebras of Theorem \ref{thm:theyareSame}, but we will wait until Section \ref{sec:theSame} to prove this. As it turns out, the desired isomorphism of complexes is quite natural. Consider a tableau $T \in \ek_{k+1}$, so that 
$$T = \ytableausetup
{boxsize=1.5em}
\begin{ytableau}
i_1 & j_1 &\cdots & j_{d} \\
i_2 \\
\vdots \\
i_{k} \\
\end{ytableau},$$
where $i_1 < \cdots < i_k < j_d$ and $j_1 \leq \cdots \leq j_d$. Then, the tableau 
$$T' = \ytableausetup
{boxsize=1.85em}
\begin{ytableau}
i_1 & j_1 &\cdots & j_{d-1} \\
i_2 \\
\vdots \\
i_k \\
j_d \\
\end{ytableau}$$
will represent a tableau in $L_d^k (F)$; indeed, the following Proposition shows that this candidate map is precisely the isomorphism we are looking for.

\begin{prop}\label{prop:EKandLIso}
Adopt notation and hypotheses of Setup \ref{set:Lcomplexsetup} with $R = k[x_1 , \dotsc , x_n]$ and $\im (\psi) = (x_1 , \dotsc , x_n)$. For each $k \geq 1$, the $R$-module homomorphisms
\begingroup\allowdisplaybreaks
\begin{align*}
    \phi_k : \ek_k &\to L(\psi , d)_k \\
    T &\mapsto  (-1)^{k-1} (T \backslash j_d)_{j_d} \\
\end{align*}
\endgroup
is an isomorphism.
\end{prop}

\begin{proof}
It suffices to show that each $\phi_k$ is a surjection, since any surjection between free modules of the same rank must be an isomorphism. Let 
$$T=\ytableausetup
{boxsize=2em}
\begin{ytableau}
i_1 & j_1 &\cdots & j_{d-1} \\
i_2 \\
\vdots \\
i_k \\
\end{ytableau}$$
be a standard tableau, so that $i_1 < \cdots < i_{k}$ and $i_1 \leq j_1 \leq \cdots \leq j_{d-1}$. If $i_k \geq j_{d-1}$, then observe that $(T \backslash i_k)^{i_k}$ is a basis element of $\ek_k$, and $\phi_k ( (T \backslash i_k)^{i_k}) =  T$. 

Assume now that $i_k < j_{d-1}$. By the relation of Observation \ref{obs:straighteningrelations}, we may write
$$T = \sum_{\ell=1}^k (-1)^{\ell} (T \backslash \{ i_\ell , j_{d-1} \} )^{i_\ell}_{j_{d-1}}.$$ 
Observe that by assumption, $(T \backslash i_\ell)^{i_\ell}$ is a basis element for $\ek_k$ for all $0 \leq \ell \leq k$. Moreover, by the above equality, 
$$\phi_k (\sum_{\ell=1}^k (-1)^{\ell+k} (T \backslash i_\ell)^{i_\ell} ) =  T,$$
so the result follows.
\end{proof}

\begin{prop}\label{prop:isoofEKandL}
Adopt notation and hypotheses of Setup \ref{set:Lcomplexsetup} with $R = k[x_1 , \dotsc , x_n]$ and $\im (\psi) = (x_1 , \dotsc , x_n)$. If $\phi_0 = \id : R \to R$, then the maps
$$\phi_\bullet : \ek_\bullet \to L ( \psi , d)$$
of Proposition \ref{prop:EKandLIso} form an isomorphism of complexes, extending the identity in homological degree $0$.
\end{prop}

\begin{proof}
In view of Proposition \ref{prop:EKandLIso}, it suffices to show that each $\phi_\bullet$ is a morphism of complexes; that is, for each $k$ the diagram
$$\xymatrix{\ek_{k+1} \ar[d]^-{\phi_{k+1}} \ar[r]^-{d_\ek} & \ek_{k} \ar[d]^-{\phi_{k}} \\
L(\psi , d)_{k+1} \ar[r]^{d_L} & L(\psi , d)_k \\}$$
commutes, where $d_L := \kos^\psi \otimes 1$. For $k=0$, this is trivial, so assume $k \geq 1$. Going clockwise around the above diagram:
\begingroup\allowdisplaybreaks
\begin{align*}
    T &\mapsto \sum_{\ell=1}^k (-1)^\ell x_{i_\ell} (T \backslash i_\ell) - \sum_{\ell=1}^k (-1)^\ell x_{j_d} (T \backslash \{ i_\ell , j_d \} )^{i_\ell} \\
    &\mapsto (-1)^{k-1} \sum_{\ell=1}^k (-1)^\ell x_{i_\ell} (T \backslash \{i_\ell , j_d \})_{j_d} - (-1)^{k-1} \sum_{\ell=1}^k (-1)^\ell x_{j_d} (T \backslash \{ i_\ell ,j_{d-1}, j_d \} )^{i_\ell}_{j_{d-1}} \\
\end{align*}
\endgroup
By Observation \ref{obs:strtrewrite}, 
$$(T \backslash \{ i_k , j_{d-1}, j_d \})^{i_k}_{j_{d-1}} = \sum_{\ell=0}^{k-1} (-1)^{k-\ell-1} (T \backslash \{ i_\ell , j_{d-1}, j_d \})^{i_\ell}_{j_{d-1}} + (-1)^k (T \backslash  j_d ),$$
whence
\begingroup\allowdisplaybreaks
\begin{align*}
    \sum_{\ell=1}^k (-1)^\ell x_{j_d} (T \backslash \{ i_\ell ,j_{d-1}, j_d \} )^{i_\ell}_{j_{d-1}} &= \sum_{\ell=1}^{k-1} (-1)^\ell x_{j_d} (T \backslash \{ i_\ell ,j_{d-1}, j_d \} )^{i_\ell}_{j_{d-1}} \\
    &+ (-1)^k x_{j_d} \sum_{\ell=0}^{k-1} (-1)^{k-\ell-1} (T \backslash \{ i_\ell , j_{d-1}, j_d \})^{i_\ell}_{j_{d-1}} \\
    & + x_{j_d} (T \backslash  j_d ) \\
    &=  x_{j_d} (T \backslash  j_d ). \\
\end{align*}
\endgroup
Substituting, this implies
\begingroup\allowdisplaybreaks
\begin{align*}
    &(-1)^{k-1} \sum_{\ell=1}^k (-1)^\ell x_{i_\ell} (T \backslash \{i_\ell , j_d \})_{j_d} - (-1)^{k-1} \sum_{\ell=1}^k (-1)^\ell x_{j_d} (T \backslash \{ i_\ell ,j_{d-1}, j_d \} )^{i_\ell}_{j_{d-1}} \\
    =& (-1)^{k-1} \sum_{\ell=1}^k (-1)^\ell x_{i_\ell} (T \backslash \{i_\ell , j_d \})_{j_d} +(-1)^k (T \backslash j_d). \\
\end{align*}
\endgroup
Moving counterclockwise around the diagram,
\begingroup\allowdisplaybreaks
\begin{align*}
    T &\mapsto (-1)^{k} (T \backslash j_d)_{j_d} \\
    &\mapsto (-1)^{k}\sum_{\ell=1}^k (-1)^{\ell+1} x_{i_\ell} (T \backslash \{ i_\ell , j_d \} )_{j_d} + (-1)^k x_{j_d} (T \backslash j_d). \\
\end{align*}
\endgroup
\end{proof}

\section{The Map $\phi$ is an Isomorphism of Algebras}\label{sec:theSame}

In this section, we prove that the map $\phi$ of Proposition \ref{prop:isoofEKandL} is an isomorphism of algebras with respect to the products of Definitions \ref{def:LDGprod} and \ref{def:DGstructureonEK}. The proof follows by first establishing that $\phi$ commutes with any product of the form $T \cdot T'$, where $T$ is an element of homological degree $1$. Since the product of elements of general homological degree can be reduced to products of the above form, the general case will follow immediately from this. To begin, we observe the following, which is immediate from the definitions referenced above:

\begin{obs}\label{obs:vanishingProd}
Let $$T = \ytableausetup
{boxsize=1.5em}
\begin{ytableau}
p_1 & p_2 &\cdots & p_d \\
\end{ytableau},\qquad T' = \ytableausetup
{boxsize=1.5em}
\begin{ytableau}
i_1 & q_2 & q_3 &\cdots & q_d \\
\vdots \\
i_k
\end{ytableau}$$ be tableaux in $L( \psi , d)$ with $p_1 \leq i_1$. Then
$$T \cdot (T')_{p_1} = 0.$$
Similarly, let $$T = \ytableausetup
{boxsize=1.5em}
\begin{ytableau}
p_1 & p_2 &\cdots & p_d \\
\end{ytableau},\qquad T' = \ytableausetup
{boxsize=1.5em}
\begin{ytableau}
i_1 & q_1 & q_2 &\cdots & q_d \\
\vdots \\
i_k
\end{ytableau}$$ be tableaux in $\ek_\bullet$ with $p_1 \leq q_1$. Then
$$T \cdot (T')_{p_1} = 0.$$
\end{obs}

The following lemma is essentially the base case of the proof of Theorem \ref{thm:theyareSame}; the proof necessarily splits into multiple cases. Recall that the notation $| \cdot |$ will denote homological degree.

\begin{lemma}\label{lem:dgFor1n}
Let $\phi : \ek_\bullet \to L(\psi , d)$ denote the map of Proposition \ref{prop:EKandLIso}. Then, for any $T, T' \in \ek_\bullet$ with $|T| = 1$, one has
$$\phi (T \cdot T' ) = \phi (T) \cdot \phi (T').$$
\end{lemma}

\begin{proof}
Let $S = R[x_1^{-1} , \dots , x_n^{-1}]$ denote the Laurent polynomial ring. It suffices to show that $\phi' := S \otimes \phi : S \otimes \ek_\bullet \to S \otimes L(\psi , d)$ is an isomorphism of algebras that restricts to an isomorphism of subalgebras $\ek_\bullet \to L(\psi , d)$. Write $$T = \ytableausetup
{boxsize=1.5em}
\begin{ytableau}
p_1 & p_2 &\cdots & p_d \\
\end{ytableau},$$
$$T' = \ytableausetup
{boxsize=1.5em}
\begin{ytableau}
i_1 & q_1 & q_2 &\cdots & q_d \\
\vdots \\
i_k
\end{ytableau}.$$
Let $0 \leq j \leq d$ be the largest integer such that $p_i = q_i$ for all $1 \leq i \leq j$; the proof of cases 1 and 2 below will follow by downward induction on $j$. In the base case $j=n$, one computes:
\begingroup\allowdisplaybreaks
\begin{align*}
    \phi ( T \cdot T' ) &= 0, \quad \textrm{and} \\
    \phi (T) \cdot \phi (T') &= x_{p_1} \cdots x_{p_{d-1}} (T')_{p_d , p_d} \\
    & = 0.
\end{align*}
\endgroup
For the inductive step, it is necessary to split the proof into cases:

\textbf{Case 1:} $p_1 \leq i_1$, $p_1 \leq q_1$, and $i_1 > q_1$ (or $k=0$). Employing Proposition \ref{prop:EKDGprod}, one computes:
\begingroup\allowdisplaybreaks
\begin{align*}
    \phi' (T \cdot T' ) &= \phi' \big( - x_{p_2} \cdots x_{p_d} (T')_{p_1} + x_{q_1} T \cdot (T' \backslash q_1)^{p_1} \big) \\
    &=  (-1)^{k} x_{p_2} \cdots x_{p_d} (T')_{p_1,q_d} +x_{q_1} \underbrace{\phi' (T) \cdot \phi' \big(  (T' \backslash q_1)^{p_1}  \big) }_{\textrm{by inductive hypothesis}}, \quad \textrm{and} \\
    \phi' (T) \cdot \phi' (T') &= (-1)^{k-1} x_{p_2} \cdots x_{p_d} (T')_{q_d , p_1} + x_{q_1} \phi (T) \cdot  \phi' \big(  (T' \backslash q_1)^{p_1}  \big) \\
    &=  (-1)^{k} x_{p_2} \cdots x_{p_d} (T')_{p_1 , q_d} + x_{q_1} \phi' (T) \cdot  \phi' \big(  (T' \backslash q_1)^{p_1}  \big) . 
\end{align*}
\endgroup
For the penultimate equality above, recall that since $d \geq 2$, one has $\phi' \big(  (T' \backslash q_1)^{p_1}  \big) = \big (\phi' (T') \backslash q_1 \big)^{p_1}$. 

\textbf{Case 2:} $p_1 \leq i_1$ and $i_1 \leq q_1$ (and $k \geq 1$). If $p_1 = i_1$, then this follows from Observation \ref{obs:vanishingProd} since both products are $0$, so assume $p_1 < i_1$. One computes:
\begingroup\allowdisplaybreaks
\begin{align*}
    \phi' (T \cdot T' ) &= \phi' \big( - x_{p_2} \cdots x_{p_d} (T')_{p_1} + x_{q_1} T \cdot (T' \backslash q_1)^{p_1} \big) \\
    &=  (-1)^{k} x_{p_2} \cdots x_{p_d} (T')_{p_1,q_d} + \underbrace{\phi' (T) \cdot \phi' \big(  (T' \backslash q_1)^{p_1}  \big) }_{\textrm{by Case 1}} \\
    &=  (-1)^{k} x_{p_2} \cdots x_{p_d} (T')_{p_1,q_d}.
\end{align*}
\endgroup
In the above, notice that $\phi' (T) \cdot \phi' \big(  (T' \backslash q_1)^{p_1}  \big) = 0$ since upon expanding $(T' \backslash q_1)^{p_1}_{q_d}$ as a linear combination of standard tableaux, each tableaux will have a $p_1$ appearing in the top corner. This means one may apply Observation \ref{obs:vanishingProd} to deduce that the product is $0$. On the other hand:
\begingroup\allowdisplaybreaks
\begin{align*}
    \phi' (T) \cdot \phi' (T') &=  (-1)^{k} x_{p_2} \cdots x_{p_d} (T')_{p_1,q_d}.
\end{align*}
\endgroup

\textbf{Cases 3 and 4:} Either $p_1 > i_1$, or $p_1 \leq i_1$, $p_1 > q_1$, and $i_1 > q_1$. This case follows by induction on the homological degree of $T'$; observe that when $|T'| = 1$, this is handled by Case 1. Let $|T'| > 1$ and define $T'' := \ytableausetup
{boxsize=1.5em}
\begin{ytableau}
i_1 & q_1 & q_2 &\cdots & q_d \\
\end{ytableau}$; observe that $|T' \backslash i_1| < |T'|$, so the inductive hypothesis applies to $T \cdot (T' \backslash q_1)$. One computes:
\begingroup\allowdisplaybreaks
\begin{align*}
    \phi' (T \cdot T') &= \frac{1}{x_{q_1} \cdots x_{q_{d-1}}} \phi' (T \cdot T'' \cdot (T \backslash i_1) ) \\
    &=\frac{-1}{x_{q_1} \cdots x_{q_{d-1}}} \underbrace{\phi' (T'')  \cdot \phi' \big( T \cdot (T' \backslash i_1) \big)}_{\textrm{by Case 1}} \\
    &= \frac{-1}{x_{q_1} \cdots x_{q_{d-1}}} \phi' (T'') \cdot \underbrace{\phi' (T) \cdot \phi' (T' \backslash i_1))}_{\textrm{By induction on} \ |T'|} \\
    &= \frac{(-1)^{k+1}}{x_{q_1} \cdots x_{q_{d-1}}}\phi' (T) \cdot T'' \cdot (T'' \backslash i_1 , q_d)_{q_d} \\
    &= \phi' (T) \cdot \phi' (T') .
\end{align*}
\endgroup
\end{proof}

Finally, we arrive at the main result of the paper:

\begin{theorem}\label{thm:theyareSame}
The map $\phi$ of Proposition \ref{prop:EKandLIso} is an isomorphism of DG-algebras.
\end{theorem}

\begin{proof}
The proof follows by induction on $|T|$. The hard part is the base case $|T|=1$, which was already done in the proof of Lemma \ref{lem:dgFor1n}. Again, let $S = R[x_1^{-1} , \dots , x_n^{-1}]$ denote the Laurent polynomial ring and $\phi' := S \otimes \phi : S \otimes \ek_\bullet \to S \otimes L(\psi , d)$. For the general case, assume that $|T| >1$ and let $$T = \ytableausetup
{boxsize=1.5em}
\begin{ytableau}
j_1 & p_1 &\cdots & p_{d} \\
\vdots \\
j_\ell
\end{ytableau} \qquad T'' = \ytableausetup
{boxsize=2.5em}
\begin{ytableau}
j_1 & p_1 &\cdots & p_{d-1} \\
\end{ytableau}, \quad \textrm{and}$$
$$T' = \ytableausetup
{boxsize=1.5em}
\begin{ytableau}
i_1 & q_1 & q_2 &\cdots & q_d \\
\vdots \\
i_k
\end{ytableau}.$$
Then,
\begingroup\allowdisplaybreaks
\begin{align*}
    \phi' (T \cdot T') &= \frac{1}{x_{p_1} \cdots x_{p_{d-1}}} T'' \cdot \phi \big( (T \backslash j_1) \cdot T' \big) \\
    &= \frac{1}{x_{p_1} \cdots x_{p_{d-1}}} T'' \cdot \underbrace{\phi' (T \backslash j_1) \cdot \phi' (T')}_{\textrm{by induction on} \ |T|} \\
     &= \frac{(-1)^{\ell + k} }{x_{p_1} \cdots x_{p_{d-1}}} T'' \cdot (T \backslash j_1 , p_d)_{p_d} \cdot (T' \backslash q_d)_{q_d} \\
    &= \phi' (T) \cdot \phi' (T').  
\end{align*}
\endgroup
\end{proof}

\bibliographystyle{amsplain}
\bibliography{biblio}
\addcontentsline{toc}{section}{Bibliography}

\end{document}